\documentclass{amsart}
\usepackage[utf8]{inputenc}
\usepackage{amssymb} 
\usepackage{nicefrac} 
\usepackage{hyperref}
\usepackage{xcolor}

\hypersetup{
    colorlinks=true,    
    linkcolor=blue,          
    citecolor=blue,      
    filecolor=blue,      
    urlcolor=blue           
}

\title{Bounding toric singularities with normalized volume}
\author{Joaqu\'in Moraga and Hendrik S\"u{\ss}}

\newtheorem{theorem}{Theorem}[section] 

\newtheorem{lemma}[theorem]{Lemma}
\newtheorem{proposition}[theorem]{Proposition}

\newtheorem{corollary}[theorem]{Corollary}

{\theoremstyle{remark}
\newtheorem{remark}[theorem]{Remark}

}

\newcommand{\qq}{\mathbb{Q}}
\newcommand{\zz}{\mathbb{Z}}

\newcommand{\rr}{\mathbb{R}}

\newcommand{\kk}{\mathbb{K}}

\newcommand{\wt}{\mathbf{wt}}

  \newtheorem{introthm}{Theorem}

  \newtheorem{introcor}{Corollary}

\theoremstyle{definition}

\newtheorem{definition}[theorem]{Definition}

\newcommand{\CC}{\mathbb{C}}

\newcommand{\RR}{\mathbb{R}}
\newcommand{\QQ}{\mathbb{Q}}
\newcommand{\ZZ}{\mathbb{Z}}
\newcommand{\NN}{\mathbb{N}}

\DeclareMathOperator{\vol}{vol}

\DeclareMathOperator{\conv}{conv}
\DeclareMathOperator{\cone}{cone}
\DeclareMathOperator{\trcone}{tr.cone}
\DeclareMathOperator{\Hom}{Hom}
\DeclareMathOperator{\Spec}{Spec}
\DeclareMathOperator{\SL}{SL}

\calclayout

\begin{document}
\begin{abstract}
We study the normalized volume of toric singularities.  As it turns out, there is a close relation to the notion of (non-symmetric) Mahler volume from convex geometry. This observation allows us to use standard tools from convex geometry, such as the Blaschke-Santal\'o inequality and Radon's theorem to prove non-trivial facts about the normalized volume in the toric setting. For example, we prove that for every $\epsilon > 0$ there are only finitely many $\QQ$-Gorenstein toric singularities with normalized volume at least $\epsilon$. From this result, it directly follows that there are also only finitely many toric Sasaki-Einstein manifolds of volume at least $\epsilon$ in each dimension.  Additionally, we show that the normalized volume of every toric singularity is bounded from above by that of the rational double point of the same dimension. Finally, we discuss certain bounds of the normalized volume in terms of topological invariants of resolutions of the singularity. We establish two upper bounds in terms of the Euler characteristic and of the first Chern class, respectively. We show that a lower bound, which was conjectured earlier by He, Seong, and Yau, is closely related to the non-symmetric Mahler conjecture in convex geometry.
\end{abstract}

\maketitle

\section{Introduction}

The study of log terminal singularities
has played a central role in the development of birational geometry, in particular, within the minimal model program~\cite{Xu17}
and for the algebraic theory of K-stability~\cite{Xu20}.
Among log terminal singularities, toric singularities are
the most well-known, as they are closely related
to orbifold singularities~\cite{CLS11,Ful93}.
The combinatorial nature of toric singularities
allows us to use methods from convex geometry to study them~\cite{CLS11}.
There are two important invariants for log terminal singularities,
the normalized volume~\cite{LLX20} and the minimal log discrepancy~\cite{Mor18b,Mor21b}.
The former is related to K-stability~\cite{Li18},
while the latter is related to termination of flips~\cite{Sho04}.
Minimal log discrepancies are well-understood for
toric singularities~\cite{Amb06}.
In~\cite{han2020acc}, the authors conjecture that 
log terminal singularities with normalized volume bounded away from zero, are bounded up to deformations.
This conjecture is closely related to work on deformations
for exceptional singularities~\cite{HLM20,Mor18c}.
Our first result is a positive answer for this conjecture
in the setting of $\mathbb{Q}$-Gorenstein toric singularities (see Corollary~\ref{cor:bounding-cones}).

\begin{introthm}\label{introthm:bounding-toric}
Let $d$ be a positive integer
and $\epsilon>0$ be a positive real number.
There are only finitely many
$d$-dimensional toric $\mathbb{Q}$-Gorenstein 
singularities with $\widehat{\vol}(X)>\epsilon$.
\end{introthm}

In particular, it is expected that the 
normalized volume of $d$-dimensional
log terminal singularities can only accumulate to zero.
As a consequence of Theorem~\ref{introthm:bounding-toric}, 
we prove this conjecture in the case of
$\mathbb{Q}$-Gorenstein toric singularities.

\begin{introcor}\label{introcor:acc}
Let $d$ be a positive integer. 
Then, the set 
\[
\mathcal{V}_{T,d}:=
\{ \widehat{\vol}(X) \mid 
\text{ $X$ is a $\mathbb{Q}$-Gorenstein
$d$-dimensional toric singularity}
\}, 
\]
only accumulates to zero. 
\end{introcor}

As a corollary, we also obtain that 
there are only finitely many toric Sasaki-Einstein manifolds
of volume at least $\epsilon$ in each dimension.

\begin{introcor}\label{introcor:SE}
Let $n$ be a positive integer
and $\epsilon>0$ be a positive real number.
There are only finitely many toric Sasaki-Einstein 
manifolds $X$
of dimension $n$ with $\vol(X)>\epsilon$.
\end{introcor}

The previous corollary is tightly related to~\cite[Theorem 1.1]{Jia20}. See also~\cite[Theorem 6.5 and Remark 6.7]{Xu20}).
It is expected that the rational double point
of dimension $d$ is the mildest
among all $d$-dimensional singularities.
This principle holds for almost all known  
invariant of singularities.
In particular, it is expected that 
rational double points of dimension $d$ 
compute the highest normalized volume 
among all log-terminal singularities.
This is conjectured by Spotti and Sun~\cite{SS17}.
The previous conjecture is known up
to dimension
three due to the work of Li, Liu, and Xu~\cite{LX19,LL19}.
In this direction, we prove the following theorem
(see Theorem~\ref{thm:rdp-bound}).

\begin{introthm}\label{introthm:higher-norm}
The normalized volume of a $d$-dimensional toric singularity
is at most $2(d-1)^d$ with equality only for the $A_1$
singularities in dimension $2$ and $3$.
\end{introthm}

In the proof of the previous theorems, 
we will compute the normalized volume
of the toric singularity by using techniques of convex geometry.
Particularly, the non-symmetric Mahler volume~\cite{RSW12}.
Well-established tools from convex geometry, such as the Blaschke-Santal\'o inequality~\cite{BK15}
and Radon's theorem will be used to prove
the boundedness theorem for toric singularities
using normalized volume and the upper bound for the normalized volume, respectively.

The paper is organized as follows.
In Section~\ref{sec:prel}, we
introduce some preliminaries related
to toric singularities and normalized volume.
In Section~\ref{sec:boundedness}, we prove
the boundedness result for toric singularities
using normalized volume.
In Section~\ref{sec:maximal}, 
we show that the normalized volume of
a toric singularity is at most the
normalized volume of a rational double point.
Finally, in Section~\ref{sec:topological}, we 
will establish some relations between topological invariants, such as the
Euler characteristic and the first Chern class, and the normalized volume of toric singularities.

\subsection*{Acknowledgement}
The authors would like to thank
Yuchen Liu, Chenyang Xu, and Ziquan Zhuang for many useful comments.
Special thanks to Benjamin Nill for helpful conversations on the subject.
After this project was completed, Ziquan informed us 
of an alternative proof of Theorem~\ref{introthm:bounding-toric}
using minimal log discrepancies
and Koll\'ar components~\cite{Zhu21}.
The second author was partly supported by the EPSRC grant EP/V013270/1.

\section{Preliminaries}
\label{sec:prel} 
In this section, we
recall some preliminaries regarding
toric singularities, normalized volume, 
and the Blasche-Santal\'o inequality.
For the basics of toric geometry we
refer the reader to~\cite{CLS11,Ful93}.

\begin{definition}
Let $N$ be a free finitely generated abelian group of rank $d$.
Let $M:=\Hom(N,\zz)$ be its dual and $\langle \cdot, \cdot \rangle$ the natural pairing.
We denote by $N_\qq$ (resp. $M_\qq$)
the associated $\qq$-vector spaces.
Analogously, we denote by
$N_\rr$ and $M_\rr$ the corresponding
$\rr$-vector spaces.
Let $\sigma\subset N_\qq$ be a 
strictly convex rational polyhedral cone.
We denote by
$\sigma^\vee \subset M_\qq$ its dual cone. 
Then, the affine variety
\[
X(\sigma):=\Spec(\kk[M\cap\sigma^\vee])
\]
is an {\em affine toric variety}.
The affine variety $X(\sigma)$ has dimension $d$.
It admits the action of the $d$-dimensional
algebraic torus
\[
\mathbb{T}:= \Spec(\kk[M]).
\] 
The affine variety $X(\sigma)$ admits a unique closed
point which is torus invariant.
We say that a singularity of a normal
variety $(X;x)$ is {\em toric} 
if it is isomorphic to
$(X(\sigma);x_0)$, 
where $X(\sigma)$ is an affine toric variety
and $x_0$ is the unique torus invariant closed point.
\end{definition}

In what follows, we are mostly concerned with toric singularities.
Then, when writing $X(\sigma)$, we mean the 
toric singularity associated to the affine toric variety, i.e., the statements that we write for singularities 
hold after possibly shrinking around the closed torus invariant point.

\begin{definition}
Let $X(\sigma)$ be a $d$-dimensional toric singularity.
The set of the extremal rays of $\sigma$ is denoted by
$\sigma(1)$.
For each extremal ray $\rho$ of $\sigma \subset N_\qq$, 
we denote by $v_\rho$ its primitive lattice generator (in the lattice $N$).
The singularity $X(\sigma)$ is Gorenstein (resp. $\qq$-Gorenstein) if and only if 
there exists $u\in M$ (resp. $u\in M_\qq$) for which 
$\langle u,v_\rho \rangle =1$ for each $\rho$.
If $X(\sigma)$ is a $\qq$-Gorenstein toric singularity,
then the minimal positive integer $\ell$ for which 
$\ell K_{X(\sigma)}\sim 0$ is called the {\em Gorenstein index} of the singularity. In the notation from above this is exactly the minimal positive integer $\ell$ such that $\ell u \in M$. 
\end{definition}

The following remark allows us to write $\qq$-Gorenstein toric singularities as cones over lattice polytopes.

\begin{remark}
Let $X(\sigma)$ be a $d$-dimensional $\qq$-Gorenstein toric singularity. Let $\ell$ be the Gorenstein index of the singularity and $u \in \frac{1}{\ell}M$ as above. Then the convex hull $P':=\conv\{v_\rho \mid \rho \in \sigma(1)\}$ lies in the hyperplane
$[\langle \cdot, u \rangle =1]:=\{v \in N_\RR \mid \langle u, v \rangle = 1\}$ and we have $\sigma=\RR_{\geq0}\cdot P'$. We have the following exact sequence of lattices
\[0 \to N_u \to N \to \ZZ \to 0,\]
where $N_u = [\langle \cdot , u \rangle =0] \cap N$. After choosing a splitting of the above sequence we obtain isomorphisms $N \cong N_u \times \ZZ$ and
$N_\RR \cong (N_u \otimes \RR) \times \RR$, respectively. With these isomorphisms we have $P' = P \times \{\ell\}$ for some lattice polytope $P\subset N_u \otimes \RR$ and we obtain
\[
\sigma = \cone(P,\ell):= \mathbb{R}_{\geq 0}\cdot (P\times \{\ell\}).
\]
Different choices for the splitting of our sequence lead to unimodularly equivalent polytopes. Hence, our choice of $P$ was unique up to unimodular equivalence.

Observe that $P$ has dimension $n=d-1$
and the singularity has dimension $d$.
We will fix this notation for the rest of the article, i.e., 
$d$ will stand for the dimension of the toric singularity (or the cone),
while $n$ will stand for the dimension of the polytope (or the corresponding projective toric variety).
\end{remark}

Now, we turn to define $\qq$-Gorenstein 
log terminal 
singularities
and their normalized volume.

\begin{definition}
Let $x\in X$ be a $\qq$-Gorenstein singularity.
Let $\pi\colon Y\rightarrow X$ be a projective birational morphism
from a normal variety $Y$.
Let $E\subset Y$ be a prime divisor.
We define the {\em log discrepancy} of $X$ with respect to $E$, to be 
\[
a_E(X):=1-{\rm coeff}(K_Y-\pi^*(K_X)).
\] 
We say that $X$ is {\em log terminal}
if $a_E(X)>0$ for every prime divisor $E$ over $X$.
We say that $x\in X$ is a {\em log terminal singularity}
if some neighborhood of $x$ in $X$ is log terminal.
For a log terminal $\qq$-Gorenstein singularity the log discrepancy 
can be regarded as a function
$a_X\colon {\rm DVal}(X)\rightarrow (0,\infty)$ with
\[
a_X(v_E):=a_E(X),
\] 
from the space of divisorial valuations ${\rm DVal}(X)$ over $X$.
In~\cite{JM12}, the authors prove that this function can be extended
to a function $a_X\colon {\rm Val}(X)\rightarrow (0,\infty]$
from the space of valuations over $X$.
We denote by ${\rm Val}(X,x)$ the space of valuations over $X$ with center $x$.
\end{definition}

The volume of a valuation over a singularity is defined in~\cite{ELS03}.

\begin{definition}
Let $x\in X$ be a normal algebraic singularity
of dimension $n$.
The {\em volume} of a valuation 
$v\in {\rm Val}(X,x)$ is defined to be
\[
\vol_X(v):=\frac{\ell(\mathcal{O}_{X,x}/a_m(v))}{m^n/n!}.
\] 
Here, $a_m(v):=\{ f\in \mathcal{O}_{X,x} \mid v(f)\geq m\}$  
and $\ell$ denotes the Artinian length of the module.
\end{definition}

Using the two previous definitions, 
we can define the normalized volume. 
The following is a special case of a definition
due to Li~\cite{Li18}.

\begin{definition}
Let $X$ be a $d$-dimensional 
$\qq$-Gorenstein log terminal singularity.
The {\em normalized volume function}
is a function 
\[
\widehat{\rm vol}\colon {\rm Val}(X,x)\rightarrow (0,\infty)
\]
defined by
\[
\widehat{\rm vol}(v)=a_X(v)^d \vol_X(v) 
\] 
if $a_X(v)<\infty$
and as $\infty$ otherwise.
The {\em normalized volume} of the
$\qq$-Gorenstein log terminal singularity $x\in X$
is defined to be
\[
\widehat{\rm vol}(X,x):=
\inf\{ \widehat{\rm vol}(v) \mid v\in {\rm Val}(X,x)\}.
\] 
\end{definition}

By the work of Blum, we know that the infimum
in the definition
of normalized volume of a singularity 
is indeed a minimum~\cite{Blu18}.
The following proposition allows to compute the normalized volume of $d$-dimensional $\qq$-Gorenstein toric singularities. 

\begin{proposition}[{\cite{zbMATH05159638},\cite[Sec~2.2.2.1 ]{li2017stability}}]
  \label{prop:vol}
  Let $X=X(\sigma)$ a $d$-dimensional $\QQ$-Gorenstein toric singularity.
  Then normalized volume $\widehat \vol (X)$ of $X$ coincides with
  the minimum of the lattice volume of a truncated cone 
  \[
  \sigma^\vee(\xi) := [\langle \xi, \cdot \rangle \leq 1] \cap \sigma^\vee
  \]
  with $\xi \in \conv \{v_\rho \mid \rho \in \sigma(1)\}$, i.e.
  \[
    \widehat \vol (X) = \min \{d!\vol \sigma^\vee(\xi) \mid
    \xi \in \conv \{ v_\rho \mid \rho \in \sigma(1)\}
    \}.
  \]
\end{proposition}

\begin{definition}
Let $P\subset \mathbb{R}^n$ be a convex body.
It's {\em polar dual} is defined to be
\[
P^*:=
\{
y \in \mathbb{R}^n \mid \langle x,y\rangle \leq 1
\text{ for all $x\in P$ } 
\}.
\]
\end{definition}

We finish this section
by recalling the Blasche-Santal\'o inequality
for the product of the volume
of a convex body and its dual.

\begin{theorem}[Blasche-Santal\'o Inequality]
  Given a convex body $K \subset \RR^n$, such that the barycentre of $K$ coincides with the origin. Then the following inequality holds
  \[
    \vol(K)\vol(K^*) \leq \omega_n^2,
  \]
  where $K^*$ denotes the polar dual of $K$ and $\omega_n$ denotes the volume of the unit ball in $\RR^n$.
\end{theorem}

\section{Proof of boundedness in the toric case}
\label{sec:boundedness}

In this section, we prove the boundedness
of toric singularities using normalized volume.
We start this section by proving some lemmata
regarding the dual polytope.

In addtion to the cone $\cone(P,\ell) \subset \RR^{n+1}$ over a polytope $P \subset \RR^n$ we will denote
the corresponding \emph{truncated cone} as follows.
\[
\trcone(P,\ell) := \conv(P \times \{\ell\} \cup \{0\})
\]
\begin{lemma}
  \label{lem:dual-and-polar}
  For a polytope $P \subset \RR^n$ containing the origin one has
  \[\cone(P,\ell)^\vee = \cone(P^*,1/\ell).\]
\end{lemma}
\begin{proof}
  This follows easily from the
  equation
  \[
    \langle (v,\ell), (u,1/\ell)\rangle = \langle v, u\rangle + 1.   
  \]
\end{proof}

\begin{lemma}
  \label{lem:vol-trcone}
  For $P \subset \RR^n$, $\sigma = \cone(P,\ell)$ and $\xi_0 = (0,\ell) \in P \times \{\ell\}$
  we have
  \[\vol \sigma^\vee(\xi_0) = \frac{1}{(n+1) \cdot \ell}\vol P^*.\]
\end{lemma}
\begin{proof}
  It follows from Lemma~\ref{lem:dual-and-polar} that $\sigma^\vee(0,\ell)=\trcone(P^*,1/\ell)$. This is a pyramid of height $1/\ell$ over the polytope $P^*$. Hence, we obtain the result from the volume formula for a pyramid.
\end{proof}

\begin{theorem}[{\cite[Thm.~3.3.]{zbMATH03673087}}]
  \label{thm:min-barycentre}
  Given an interior point $u \in \sigma^\vee$. Then volume of a truncated cone
  $$[\langle \xi , \cdot \rangle \leq 1] \cap \sigma^\vee$$ is minimized among all
  choices $\xi$ such that $\langle \xi, u \rangle =1$ if and only if
  $u$ coincides with the barycenter of $[\langle \xi, \cdot \rangle=1] \cap \sigma^\vee$.
\end{theorem}

\begin{theorem}
  \label{thm:volume-bound}
  Assume $X=X(\sigma)$ is a $d$-dimensional toric singularity of Gorenstein index $\ell$. Then the following inequalities hold
    \begin{equation}
    \frac{1}{(d-1)!} \leq
    \vol \left(\conv \{v_\rho \mid \rho \in \sigma(1)\}\right)
    < \frac{\omega_{d-1}^2}{\ell \cdot d \cdot \widehat \vol (X)}.\label{eq:claim}  
  \end{equation}
  In particular, we have
  \begin{equation}
    \ell < \frac{(d-1)! \cdot \omega_{d-1}^2}{d \cdot \widehat \vol(X)}.\label{eq:idx-bound}
  \end{equation}
  
\end{theorem}
\begin{proof}
  Having Gorenstein index $\ell$ implies that for an appropriate choice
  of coordinates we have $\sigma = \cone(P,\ell)$ and
  $\conv \{v_\rho \mid \rho \in \sigma(1)\} = P \times \{\ell\}$.

  Now, assume $\xi=(\chi,\ell)$ minimizes $\vol \sigma^\vee(-)$.
  We consider
  \[T\colon \RR^{d-1}\times \RR \to \RR^{d-1} \times \RR;\; (v,\,h) \mapsto (v-h\cdot \chi/\ell,\,h)\]
  Then we set \[\sigma_0 = T\sigma = \cone(P_0, \ell)\]
  with $P_0=P-\chi$. Since,  $T \in \SL(n,\RR)$ preserves volumes,
  we see that $\xi_0=T\xi=(0,\ell)$ minimizes $\vol \sigma_0^\vee(-)$
  among all elements of $P_0 \times \{\ell\}$ and
  \[\vol (\sigma_0^\vee(\xi_0))=\vol ((T\sigma)^\vee(T\xi))=\vol (\sigma^\vee(\xi)).\]
  Now, by Lemma~\ref{lem:dual-and-polar} we have
  $\sigma_0^\vee = \cone(P_0^*,1/\ell)$ and obviously for all
  $v \in \RR^{n-1}$ we have $\langle (v,\ell), (0,1/\ell) \rangle = 1$.
  Since $\xi_0$ was a minimizer for $\vol \sigma_0^\vee(-)$, Theorem~\ref{thm:min-barycentre} implies that $u=0$ is the barycentre of $P_0$.
  Now, we can apply the Blaschke-Santal\'o inequality and obtain
  \[\vol P_0 \cdot \vol P_0^* < \omega_{d-1}^2.\]
  By Proposition~\ref{prop:vol} we have
  $\widehat \vol (X) = \vol (\sigma^\vee(\xi)) = \vol(\sigma_0^\vee(\xi_0)) $. 
  Then the inequality
  \[\vol P = \vol P_0
    < \frac{\omega_{d-1}^2}{\ell \cdot d \cdot \widehat \vol (X)}\]
    follows from Lemma~\ref{lem:vol-trcone}.
    Since $P$ is a lattice polytope it has volume at least $\frac{1}{(d-1)!}$
    giving rise to the inequality on the left-hand-side of (\ref{eq:claim}).
\end{proof}

\begin{remark}
 In convex geometry the element $\chi \in P$ from the proof of Theorem~\ref{thm:volume-bound} is known as the \emph{Santal\'o point} of $P$, see e.g. \cite{Santalo1949}. For a simplex the Santal\'o point coincides with the barycenter. In the following we will denote the dual $(P-\chi)^*$ of the translated polytope $(P-\chi)$ simply by $P^\chi$.
\end{remark}

\begin{corollary}
\label{cor:bounding-cones}
  In each dimension and for any $\epsilon > 0$ there are only finitely many
  toric $\QQ$-Gorenstein singularities with
  $\widehat \vol (X) > \epsilon$.
\end{corollary}
\begin{proof}
  By claim (\ref{eq:idx-bound}) of Theorem~\ref{thm:volume-bound}
  it follows that the Gorenstein index of $X$ is bounded from above
  by 
  \[
  \frac{(d-1)! \cdot \omega_{d-1}^2}{d \cdot \epsilon}.
  \]
  For a fixed Gorenstein index $\ell$ every toric singularity of that index
  is uniquely (up to unimodular equivalence) determined by a lattice polytopes $P$ via $X=X(\cone(P,\ell))$.
  Then by (\ref{eq:claim}) of Theorem~\ref{thm:volume-bound} we have
  \[\vol P < \frac{\omega_{d-1}^2\cdot (d-1)! \cdot \epsilon^{-1}}{\ell \cdot d}.\]
  Now, by \cite[Thm.~2]{zbMATH00059646} it follows that there are only
  finitely many equivalence classes of polytopes with such a 
  bounded volume.
\end{proof}

The following result is merely a reformulation of Corollary~\ref{cor:bounding-cones} in terms of Sasaki-Einstein geometry. For the details on the connections to Sasaki-Einstein geometry, we refer the reader to \cite[Sec.~3.3]{LLX20}.

\begin{corollary}
\label{cor:bounding-manifolds}
  In each dimension and for any $\epsilon > 0$ there are only finitely many toric Sasaki-Einstein manifolds of volume at least $\epsilon$.
\end{corollary}

Note also, that the Reeb field of a Sasaki-Einstein metric (as the unique minimizer of the normalized volume) in the toric setting directly corresponds to the notion of the Santal\'o point of $P$ from convex geometry.

\section{Singularities of maximal volume}
\label{sec:maximal}

We recall the notation from the previous section. We consider a polytope $P$ and the cone $\sigma$ spanned by $P \times \{\ell\}$. Let $\xi=(\chi,\ell) \in (\operatorname{relint}P) \times \{\ell\}$ and set $P^{\chi} = (P-\chi)^*$. Note that $\chi$ is the unique minimizer of $\vol P^{\chi}$ if and only if $\xi=(\chi,\ell)$ is the unique minimizer of $\vol \sigma^\vee(\xi)$.
This is also equivalent to the fact that the barycenter of $P^{\chi}$ coincides with the origin. 

\begin{lemma}
\label{lem:simplex}
Let $\Delta$ be a simplex of dimension $n$ and $z \in \Delta$ the barycentre, then we get the following value for the so-called \emph{(nonsymmetric) Mahler volume} of $\Delta$.
\[\vol(\Delta)\vol(\Delta^{z}) = \frac{(n+1)^{(n+1)}}{(n!)^2}.\]
\end{lemma}
\begin{proof}
This can be easily checked for the standard simplex.
Then, the statement follows from the fact that the Mahler volume is invariant under affine transformations.
\end{proof}

\begin{lemma}
\label{lem:polytope}
Let $P$ be a lattice polytope and $\chi \in P$ the Santal\'o point. Then we have
\[\vol(P^\chi) \leq \frac{(n+1)^{(n+1)}}{n!}\]
with equality only for the standard simplex.
\end{lemma}
\begin{proof}
We may consider any lattice simplex $\Delta \subset P$. Let $z$ be the barycentre of $\Delta$. Then we have $P^z \subset \Delta^z$ and therefore
\[\vol(P^\chi) \leq \vol(P^z) \leq \vol(\Delta^z) \leq  \frac{(n+1)^{(n+1)}}{n!}.\]
Here, the rightmost inequality follows from Lemma~\ref{lem:simplex} and the fact that for a lattice simplex $\Delta$ we have $\vol(\Delta)\geq 1/n!$. 
On the one hand, if $P=\Delta$, but not the standard simplex we have now $\vol(\Delta)\geq 2/n!$. On the other hand, we have $\vol P^z < \vol \Delta^z$ if $P \subsetneq \Delta$. Hence, in either case, we get a strict inequality.
\end{proof}

\begin{proposition}
\label{prop:singular-bound}
Every toric singularity of dimension $d$ (excluding the smooth point) has normalized volume less than
\(d^d\).
\end{proposition}
\begin{proof}
Let $X=X(\sigma)$ be the $\QQ$-Gorenstein toric singularity corresponding to the cone $\sigma=\cone(P,\ell)$, where $\dim(P)=n:=d-1$. Then by Proposition~\ref{prop:vol} and Lemma~\ref{lem:vol-trcone} we have $\widehat{\vol}(X)=\frac{n!}{\ell}\vol P^\chi$, where $\chi \in P$ is the Santal\'o point. Now, the result follows from Lemma~\ref{lem:polytope}.
\end{proof}

The following theorem strengthens the result from above.

\begin{theorem}
\label{thm:rdp-bound}
Every toric singularity of dimension $d$ has
normalized volume at most $2(d-1)^d$ and equality holds only for the toric $A_1$
singularities in dimension $2$ and $3$, i.e., quadric cones.
\end{theorem}
To prove this theorem we once again need some well-known results from convex geometry.
\begin{lemma}
\label{lem:decompose-n+2}
Let $P \subset \RR^n$ be a polytope of dimension $n$ given as convex hull of $n+2$ points. Then $P = \conv(\Delta \cup \Delta')$, such that $\Delta$ and $\Delta'$ are simplices of dimensions $p$ and $q$, respectively, spanning complementary affine subspaces and $\Delta \cap \Delta'$ consists of exactly one point.
\end{lemma}
\begin{proof}
This is a consequence of Radon's Theorem, see also the proof of \cite[Thm. 5.3]{alex19}.
\end{proof}

In the setting of Lemma~\ref{lem:decompose-n+2}, we say that $P$ has a \emph{$(p,q)$-partition}. Note, that this notion is symmetric in $p$ and $q$. The unique intersection point of Lemma~\ref{lem:decompose-n+2} will be called the \emph{Radon point} of the partition.

\begin{lemma}
\label{lem:lower-bound}
Let $P$ be a lattice polytope as in Lemma~\ref{lem:decompose-n+2}. Then 
$$\vol(P) \geq \frac{\min\{p \mid P \text{ has a $(p,q)$-partition}\}+1}{n!}.$$
\end{lemma}
\begin{proof}
Assume that $p \in \NN$ is minimal such that $P$ admits a $(p,q)$-partition. Assume that such a partition is given by the two complementary simplices $\Delta$ and $\Delta'$, which intersect in the Radon point $o$. Then we obtain a triangulation of $P$ as follows. For every facet $F \prec \Delta$ we consider $Q_F=\conv (F \cup\{o\})$ and $P_F=\conv (F \cup {\Delta'})=\conv(Q_F \cup \Delta')$

Then  $\{Q_F\}_F$ and $\{P_F\}_F$ induce triangulations of $\Delta$ and $P$, respectively. We claim that all the $P_F$ are of full dimension. Indeed, assume this is not the case. Then, on the one hand, one of the $P_F$ spans a proper affine subspace of $\RR^n$ and, on the other hand, $F \subset P_F \cap \Delta$. Hence, we actually have 
$F = P_F \cap \Delta$. We conclude for the Radon point that
$$o \in \Delta \cap P_F = F.$$
But then, we also have a $(p-1,q+1)$-partition of $P$, given by $F$ and $\conv(\Delta' \cup \{v_F\})$, where $v_F$ is the unique vertex of $\Delta$, which is not contained in $F$. Hence, $p$ was not minimal. This leads to a contradiction.

Now, we are given a triangulation of $P$ into $(p+1)$ full dimensional lattice simplices. This implies that $\vol(P) \geq (p+1)/n!$
\end{proof}

\begin{proposition}
\label{lem:upper-bound}
Given a polytope $P$ as the convex hull of $n+2$ lattice points and $p \geq 1$ is minimal such that a $(p,q)$-partition exists. Let $\chi \in P$ be the Santal\'o point. Then, we have
\[\vol(P^\chi) \leq \frac{(p+1)^{p}(q+1)^{(q+1)}}{p!q!}.\]
The equality holds if and only if the Radon point of the $(p,q)$-partition coincides with the barycentres of the two simplices.
\end{proposition}
\begin{proof}
This follows immediately from Theorem~5.1 in \cite{alex19} and Lemma~\ref{lem:lower-bound}.
\end{proof}

\begin{lemma}
\label{lem:simple-ineq}
For $1 \leq p < n$ and $q=n-p$ one has
\[\frac{(p+1)^{p}(q+1)^{(q+1)}}{p!q!} \leq \frac{2n^{(n+1)}}{n!}\]
with equality only for $p=1$.
\end{lemma}
\begin{proof}
The claim is equivalent to 
\[\binom{n}{p} = \frac{n!}{p!q!} \leq \frac{2n^{(n+1)}}{(p+1)^{p}(q+1)^{(q+1)}} =: f(n,p)\]
Note that we have $f(n,p)=n=\binom{n}{p}$ for $p=1$. We now argue by induction on $p$.
We have $\binom{n}{p+1}$ = $\binom{n}{p}\cdot \frac{q}{p+1}$ and
\[f(n,p+1)=f(n,p) \cdot \frac{q}{(p+2)\left(\frac{p+2}{p+1}\right)^p\left(\frac{q}{q+1}\right)^{q+1}}.\]
By induction hypothesis it is enough to show that
\[
\frac{q}{p+1} < \frac{q}{(p+2)\left(\frac{p+2}{p+1}\right)^p\left(\frac{q}{q+1}\right)^{q+1}}
\]
for $1 \leq p < n$ or equivalently
\begin{equation}
\left(\frac{p+2}{p+1}\right)^{p+1}\left(\frac{q}{q+1}\right)^{q+1} \leq 1. \label{eq:pre-agm}
\end{equation}
To obtain the latter observe that
\[
\sqrt[n+2]{\left(\frac{p+2}{p+1}\right)^{p+1}\left(\frac{q}{q+1}\right)^{q+1}} 
< \frac{(p+1)\cdot\frac{p+2}{p+1} + (q+1)\cdot\frac{q}{q+1}}{n+2} = \frac{n+2}{n+2} =1
\]
holds by the inequality between arithmetic and geometric mean.
\end{proof}

\begin{proposition}
\label{prop:rdp-polytope}
Assume that $P$ is a lattice polytope of dimension $n$, but not the standard simplex. 
With $\chi \in P$ being the Santal\'o point of $P$ the inequality
\[\vol(P^\chi) \leq \frac{2n^{(n+1)}}{n!}\]
holds, with equality only for the line segment of length $2$ and for the unit square.
\end{proposition}
\begin{proof}
If $P$ is a simplex (but $\vol(P)\geq 2/n!$), then the claim follows from Lemma~\ref{lem:simplex}. If it is not a simplex, then $P$ contains a lattice polytope $Q$ spanned by $n+2$ lattice points. If $\vol(Q)>2/n!$, then it follows from Lemma~\ref{lem:upper-bound} and Lemma~\ref{lem:simple-ineq} that 
\[\vol P^\chi \leq \vol P^z \leq \vol Q^z < \frac{2n^{(n+1)}}{n!}.\]
Here $z \in Q$ denotes the Santal\'o point of $Q$.

Now, if $\vol(Q)=2/n!$, then by \cite{zbMATH06618525} (see also \cite[Thm. 1.3]{Hibi_2020}) $Q$ is an iterated pyramid over the unit square. If we are in dimension $n=2$, i.e. $Q$ is the unit square, then either $P=Q$ and we indeed get equality or if $Q \subsetneq P$, then 
\[\vol P^\chi \leq \vol P^z < \vol Q^z = \frac{2n^{(n+1)}}{n!}.\]

Now, assume that $n>2$ and $Q$ is an iterated pyramid. Then by Lemma~\ref{lem:ipyramid}, we have 
\[\vol P^\chi \leq \vol P^z \leq \vol Q^z < \frac{2n^{(n+1)}}{n!}.\]
Hence, the claim follows also in this case.
\end{proof}

\begin{lemma}
\label{lem:ipyramid}
We have the following bound for the Mahler volume of an iterated pyramid $P$ over the unit square of dimension $n$.
\[\vol(P)\vol(P^\chi) \leq \frac{4n^{(n+1)}}{(n!)^2}\]
with equality exactly for the unit square itself.
\end{lemma}
\begin{proof}
Let $Q\subset \RR^2$ be the unit square. Then $Q=\conv(\Delta \cup \Delta')$ with $\Delta=\conv(0,e_1+e_2)$ and $\Delta'=\conv \{e_1,e_2\}$. Here, the intersection point coincides with the barycentre of $\Delta$ and $\Delta'$. This is not longer the case for proper pyramids over the square. Indeed, for 
$P=\conv Q \times \{0\}  \cup \{e_3, \ldots e_n\}$ we have a $(1,n-1)$-partition given by $\nabla=\Delta \times \{0\}$ and $\nabla'=\conv \Delta' \cup \{e_3,\ldots,e_n\}$, but here their intersection $(\nicefrac{1}{2},\nicefrac{1}{2},0,\ldots,0)$ does clearly not coincide with the barycentre of $\nabla'$. Now, the strict inequality follows from \cite[Thm. 5.1]{alex19}.
\end{proof}
\begin{proof}[Alternative proof] For an alternative proof we may directly calculate the Mahler volume of $P$ as $\frac{32(n+1)^{(n+1)}}{27(n!)^2}$. In particular, we see that this is at most $\frac{4n^{(n+1)}}{(n!)^2}$ for $n \geq 2$ with an equality only for $n=2$.
\end{proof}

\begin{proof}[Proof of Theorem~\ref{thm:rdp-bound}]
We argue as in the proof of Proposition~\ref{prop:singular-bound} but this time with the help of Proposition~\ref{prop:rdp-polytope} instead of Proposition~\ref{lem:polytope}. 
\end{proof}

\section{Topological and invariants and the normalized volume}
\label{sec:topological}

In \cite{zbMATH06921694} the authors try to find bounds on the normalized volume of a toric  Gorenstein singularity in terms of topological invariants of certain (crepant) resolutions. In the toric setting we suggest the following bound
\begin{equation}
    \widehat{\vol}(X) \geq \frac{d^d}{\chi(\widetilde{X})} \label{conj1}
\end{equation}
which is inspired by and closely related to \cite[Conjecture 5.3 ]{zbMATH06921694}. Here, $\chi(\widetilde{X})$ denotes the Euler characteristic of a toric resolution of singularities for $X$.  We remark that this bound would follow from the non-symmetric Mahler conjecture \cite{zbMATH03031877}, which states that for every convex body $P \subset \mathbb{R}^n$ with Santal\'o point $\chi$ the inequality
\[\vol(P)\vol(P^\chi) \geq \frac{(n+1)^{(n+1)}}{(n!)^2}\]
holds. We stress that the non-symmetric Mahler conjecture is still open. To see the implication, we assume as before that the toric singularity is given by the cone  $\sigma=\cone(P,1)$ with $\dim(P)=n:=\dim(X)-1$. Now, a toric resolution corresponds to a regular triangulation of that cone, i.e. a subdivision into simplicial cones each of them being spanned by elements of a lattice basis of $\ZZ^{n+1}$, see e.g. \cite[Sec.~2.6]{Ful93}. Such  a triangulation of $\sigma$ induces a triangulation of the polytope $P$ with facets of volume at most $1/n!$. Hence, $n!\vol(P)$ is a lower bound for the number of maximal cones in the triangulation of $\sigma$. On the other hand, the number of maximal cones (and therefore of torus fixed points in the resolution) coincides with the Euler characteristic $\chi(\widetilde{X})$ \cite[Cor.~11.6]{zbMATH03661497}. Hence, we obtain
\[\widehat{\vol}(X)\chi(\widetilde{X}) \geq n!\vol(P^\chi)n!\vol(P).\]
Under assumption of the Mahler conjecture the right-hand-side is at least $d^d=(n+1)^{(n+1)}$ and the conjectured inequality (\ref{conj1}) would follow. 

We also remark that the Blaschke-Santal\'o inequality provides an upper bound for the normalized volume of a toric singularity in terms of the Euler characteristic of a crepant (partial) resolution. Indeed, a toric crepant (partial) resolution $\widetilde{X}$ corresponds to a subdivision of $\sigma$ which induces a subdivision of $P$ into \emph{lattice} polytopes of volume at least $1/n!$. Hence, $\chi(\widetilde{X}) \leq n!\vol(P)$ in this case. Now, from the inequality (\ref{eq:claim}) in Theorem~\ref{thm:volume-bound} we obtain
\begin{equation*}
    \widehat{\vol}(X) <  \frac{(d-1)!\omega_{d-1}^2}{\chi(\widetilde{X})}.
\end{equation*}
The authors of \cite{zbMATH06921694} also speculate about a possible upper bound in terms of the first Chern class of $\widetilde{Y}$ for the case that $P$ is a reflexive polytope, $Y$ is the Gorenstein Fano variety corresponding to its face fan and $\widetilde{Y}$ is a crepant resolution of $Y$. They conjecture that in this case
\begin{equation}
\widehat{\vol}(X) \leq \int_{\widetilde{Y}} c_1(\widetilde{Y})^{d-1} \label{eq:c1}
\end{equation}
holds, see \cite[Conjecture~5.5]{zbMATH06921694}. A direct argument shows that the inequality~(\ref{eq:c1}) holds more generally for any cone $X$ over an anti-canonical embedding of a Gorenstein Fano variety $Y$ of dimension $d-1$ with crepant resolution $\widetilde{Y}$ and we have an equality if and only if $Y$ is K-semistable.
Indeed, we have  
\[\int_{\widetilde{Y}} c_1(\widetilde{Y})^{d-1}= (-K_{\widetilde{Y}})^{d-1}=(-K_Y)^{d-1} = \vol_{\wt}(X).\]
The equality in the middle follows from the fact that $\widetilde{Y}\to Y$ is crepant. The valuation $\wt$ is the one given by the weight with respect to the $\CC^*$-action coming with the cone structure of $X$. Moreover, for the anti-canonical embedding the log discrepancy $a_X(\wt)$ is $1$. This follows e.g. from Step~3 of the proof of Theorem~1 in \cite[Sec.~3]{Mor18c} (in the notation of this paper our assumption assures $D_i=-K_Y$). Hence, we obtain
\[\int_{\widetilde{Y}} c_1(\widetilde{Y})^{d-1}= \vol_{\wt}(X) \geq \widehat{\vol}(X) := \inf_v \vol_v (X).\]
Here, $v$ runs over all valuation on $X$ with center at the vertex and log discrepancy $a_X(v)=1$. 

\bibliographystyle{habbrv}
\bibliography{volbound}

\begin{thebibliography}{10}
\expandafter\ifx\csname url\endcsname\relax
  \def\url#1{\texttt{#1}}\fi
\expandafter\ifx\csname doi\endcsname\relax
  \def\doi#1{\burlalt{doi:#1}{http://dx.doi.org/#1}}\fi
\expandafter\ifx\csname urlprefix\endcsname\relax\def\urlprefix{URL }\fi
\expandafter\ifx\csname href\endcsname\relax
  \def\href#1#2{#2}\fi
\expandafter\ifx\csname burlalt\endcsname\relax
  \def\burlalt#1#2{\href{#2}{#1}}\fi

\bibitem{alex19}
M.~Alexander, M.~Fradelizi, and A.~Zvavitch.
\newblock Polytopes of maximal volume product.
\newblock {\em {}Discrete. Comput. Geom.}, 2019.

\bibitem{Amb06}
F.~Ambro.
\newblock The set of toric minimal log discrepancies.
\newblock {\em Cent. Eur. J. Math.}, 4(3):358--370, 2006.
\newblock \doi{10.2478/s11533-006-0013-x}.

\bibitem{BK15}
G.~Bianchi and M.~Kelly.
\newblock A {F}ourier analytic proof of the {B}laschke-{S}antal\'{o}
  inequality.
\newblock {\em Proc. Amer. Math. Soc.}, 143(11):4901--4912, 2015.
\newblock \doi{10.1090/proc/12785}.

\bibitem{Blu18}
H.~Blum.
\newblock Existence of valuations with smallest normalized volume.
\newblock {\em Compos. Math.}, 154(4):820--849, 2018.
\newblock \doi{10.1112/S0010437X17008016}.

\bibitem{CLS11}
D.~A. Cox, J.~B. Little, and H.~K. Schenck.
\newblock {\em Toric varieties}, volume 124 of {\em Graduate Studies in
  Mathematics}.
\newblock American Mathematical Society, Providence, RI, 2011.
\newblock \doi{10.1090/gsm/124}.

\bibitem{zbMATH03661497}
V.~I. {Danilov}.
\newblock {Geometry of toric varieties}.
\newblock {\em {Russ. Math. Surv.}}, 33(2):97--154, 1978.
\newblock \doi{10.1070/RM1978v033n02ABEH002305}.

\bibitem{ELS03}
L.~Ein, R.~Lazarsfeld, and K.~E. Smith.
\newblock Uniform approximation of {A}bhyankar valuation ideals in smooth
  function fields.
\newblock {\em Amer. J. Math.}, 125(2):409--440, 2003.
\newblock
  \urlprefix\url{http://muse.jhu.edu/journals/american_journal_of_mathematics/v125/125.2ein.pdf}.

\bibitem{zbMATH06618525}
A.~{Esterov} and G.~{Gusev}.
\newblock {Multivariate Abel-Ruffini}.
\newblock {\em {Math. Ann.}}, 365(3-4):1091--1110, 2016.
\newblock \doi{10.1007/s00208-015-1309-6}.

\bibitem{Ful93}
W.~Fulton.
\newblock {\em Introduction to toric varieties}, volume 131 of {\em Annals of
  Mathematics Studies}.
\newblock Princeton University Press, Princeton, NJ, 1993.
\newblock \doi{10.1515/9781400882526}.
\newblock The William H. Roever Lectures in Geometry.

\bibitem{zbMATH03673087}
S.~{Gigena}.
\newblock {Integral invariants of convex cones.}
\newblock {\em {J. Differ. Geom.}}, 13:191--222, 1978.

\bibitem{HLM20}
J.~Han, J.~Liu, and J.~Moraga.
\newblock Bounded deformations of {$(\epsilon,\delta)$}-log canonical
  singularities.
\newblock {\em J. Math. Sci. Univ. Tokyo}, 27(1):1--28, 2020.

\bibitem{han2020acc}
J.~Han, Y.~Liu, and L.~Qi.
\newblock Acc for local volumes and boundedness of singularities, 2020,
  \burlalt{2011.06509}{http://arxiv.org/abs/2011.06509}.

\bibitem{zbMATH06921694}
Y.-H. {He}, R.-K. {Seong}, and S.-T. {Yau}.
\newblock {Calabi-Yau volumes and reflexive polytopes}.
\newblock {\em {Commun. Math. Phys.}}, 361(1):155--204, 2018.
\newblock \doi{10.1007/s00220-018-3128-6}.

\bibitem{Hibi_2020}
T.~Hibi and A.~Tsuchiya.
\newblock Classification of lattice polytopes with small volumes.
\newblock {\em Journal of Combinatorics}, 11(3):495–509, 2020.
\newblock \doi{10.4310/joc.2020.v11.n3.a4}.

\bibitem{Jia20}
C.~Jiang.
\newblock Boundedness of {$\Bbb Q$}-{F}ano varieties with degrees and
  alpha-invariants bounded from below.
\newblock {\em Ann. Sci. \'{E}c. Norm. Sup\'{e}r. (4)}, 53(5):1235--1248, 2020.
\newblock \doi{10.24033/asens.244}.

\bibitem{JM12}
M.~Jonsson and M.~Musta\c{t}\u{a}.
\newblock Valuations and asymptotic invariants for sequences of ideals.
\newblock {\em Ann. Inst. Fourier (Grenoble)}, 62(6):2145--2209 (2013), 2012.
\newblock \doi{10.5802/aif.2746}.

\bibitem{zbMATH00059646}
J.~C. {Lagarias} and G.~M. {Ziegler}.
\newblock {Bounds for lattice polytopes containing a fixed number of interior
  points in a sublattice.}
\newblock {\em {Can. J. Math.}}, 43(5):1022--1035, 1991.

\bibitem{Li18}
C.~Li.
\newblock Minimizing normalized volumes of valuations.
\newblock {\em Math. Z.}, 289(1-2):491--513, 2018.
\newblock \doi{10.1007/s00209-017-1963-3}.

\bibitem{LL19}
C.~Li and Y.~Liu.
\newblock K\"{a}hler-{E}instein metrics and volume minimization.
\newblock {\em Adv. Math.}, 341:440--492, 2019.
\newblock \doi{10.1016/j.aim.2018.10.038}.

\bibitem{LLX20}
C.~{Li}, Y.~{Liu}, and C.~{Xu}.
\newblock {A guided tour to normalized volume}.
\newblock In {\em Geometric analysis. In honor of Gang Tian's 60th birthday},
  pages 167--219. Cham: Birkh\"auser, 2020.
\newblock \doi{10.1007/978-3-030-34953-0-10}.

\bibitem{li2017stability}
C.~Li and C.~Xu.
\newblock Stability of valuations: higher rational rank.
\newblock {\em Peking Mathematical Journal}, pages 1--79, 2017.

\bibitem{LX19}
Y.~Liu and C.~Xu.
\newblock K-stability of cubic threefolds.
\newblock {\em Duke Math. J.}, 168(11):2029--2073, 2019.
\newblock \doi{10.1215/00127094-2019-0006}.

\bibitem{zbMATH03031877}
K.~{Mahler}.
\newblock {Ein Minimalproblem f\"ur konvexe Polygone}.
\newblock {\em {Mathematica B, Zutphen}}, 7:118--127, 1938.

\bibitem{zbMATH05159638}
D.~{Martelli}, J.~{Sparks}, and S.-T. {Yau}.
\newblock {The geometric dual of $a$-maximisation for toric Sasaki-Einstein
  manifolds.}
\newblock {\em {Commun. Math. Phys.}}, 268(1):39--65, 2006.
\newblock \doi{10.1007/s00220-006-0087-0}.

\bibitem{Mor18c}
J.~Moraga.
\newblock A boundedness theorem for cone singularities, 2018,
  \burlalt{arXiv:1812.04670}{http://arxiv.org/abs/arXiv:1812.04670}.

\bibitem{Mor18b}
J.~Moraga.
\newblock On minimal log discrepancies and {K}oll\'ar components, 2018,
  \burlalt{arXiv:1810.10137}{http://arxiv.org/abs/arXiv:1810.10137}.

\bibitem{Mor21b}
J.~Moraga.
\newblock Minimal log discrepancies of regularity one, 2021,
  \burlalt{arXiv:2108.03677}{http://arxiv.org/abs/arXiv:2108.03677}.

\bibitem{RSW12}
S.~Reisner, C.~Sch\"{u}tt, and E.~M. Werner.
\newblock Mahler's conjecture and curvature.
\newblock {\em Int. Math. Res. Not. IMRN}, 1(1):1--16, 2012.
\newblock \doi{10.1093/imrn/rnr003}.

\bibitem{Santalo1949}
L.~Santaló.
\newblock Un invariante afin para los cuerpos convexos del espacio de n
  dimensiones.
\newblock {\em Portugaliae mathematica}, 8(4):155--161, 1949.
\newblock \urlprefix\url{http://eudml.org/doc/114682}.

\bibitem{Sho04}
V.~V. Shokurov.
\newblock Letters of a bi-rationalist. {V}. {M}inimal log discrepancies and
  termination of log flips.
\newblock {\em Tr. Mat. Inst. Steklova}, 246(Algebr. Geom. Metody, Svyazi i
  Prilozh.):328--351, 2004.

\bibitem{SS17}
C.~Spotti and S.~Sun.
\newblock Explicit {G}romov-{H}ausdorff compactifications of moduli spaces of
  {K}\"{a}hler-{E}instein {F}ano manifolds.
\newblock {\em Pure Appl. Math. Q.}, 13(3):477--515, 2017.
\newblock \doi{10.4310/pamq.2017.v13.n3.a5}.

\bibitem{Xu17}
C.~Xu.
\newblock Interaction between singularity theory and the minimal model program,
  2017, \burlalt{arXiv:1712.01041}{http://arxiv.org/abs/arXiv:1712.01041}.

\bibitem{Xu20}
C.~Xu.
\newblock K-stability of fano varieties: an algebro-geometric approach, 2020,
  \burlalt{2011.10477}{http://arxiv.org/abs/2011.10477}.

\bibitem{Zhu21}
Z.~Zhuang.
\newblock On boundedness of singularities and minimal log discrepancies of
  koll\'ar components, 2021.
\newblock In Preparation.

\end{thebibliography}

\end{document}